\documentclass{amsart}
\usepackage{amsthm}
\usepackage{amsaddr}

\usepackage[english]{babel}

\usepackage{color, soul}

\usepackage[retainorgcmds]{IEEEtrantools}

\usepackage[backend=bibtex8, isbn=false, url = false, doi = false, eprint = true, style=numeric, firstinits=true]{biblatex}

\DefineBibliographyStrings{english}{pages={}}

\renewbibmacro{in:}{%
  \ifentrytype{article}{}{%
  \printtext{\bibstring{in}\intitlepunct}}}

\renewcommand*{\intitlepunct}{}

\sethlcolor{yellow}

\renewcommand{\hl}{}

\newcommand{\Rb}{\mathbb R}

\newcommand{\Zb}{\mathbb Z}

\newcommand{\eps}{\varepsilon}

\newcommand{\lb}{\left(}
\newcommand{\rb}{\right)}
\newcommand{\Ac}{\mathcal{A}}

\newcommand{\sumd}[2]{\sum_{u = #1}^{#2}}

\newcommand{\enbrace}[1]{\lb #1 \rb}
\newcommand{\inv}[1]{{#1}^{-1}}
\newcommand{\setdef}[2]{\left\{ #1\ \left|\ #2 \right.\right\}}
\newcommand{\seqz}[2]{\left\{ {#1}_{#2} \right\}_{#2\in\Zb}}
\newcommand{\seq}[1]{ \{ #1 \} }

\newcommand{\half}[1]{\frac{#1}{2}}

\newcommand{\vmod}[1]{\left| #1 \right|}

\newcommand{\Zplus}{\Zb^+}
\newcommand{\Zminus}{\Zb^-}

\newcommand{\Nc}{\mathcal{N}}

\newcommand{\normop}[1]{\left\| #1 \right\|}
\newcommand{\normban}[1]{\left\| #1 \right\|}

\newcommand{\sfm}[2]{X_{#1} P X_{-#2} }
\newcommand{\ufm}[2]{X_{#1} \enbrace{I-P} X_{-#2} }

\newcommand{\ek}{\exp_{p_k}}
\newcommand{\eki}{\exp_{p_k}^{-1}}
\newcommand{\emk}{\exp^{-1}_{p_{k+1}}}

\newcommand{\Nint}{[-N,N-1]} 
\newcommand{\Nints}{[-N,N]} 

\newcommand{\rayp}{{\Zplus}}
\newcommand{\rayn}{{\Zminus}}

\numberwithin{equation}{section}

\theoremstyle{plain} \newtheorem{theorem}{Theorem}
\theoremstyle{plain} \newtheorem*{theoremnon}{Theorem}
\theoremstyle{plain} \newtheorem{statement}{Statement}[section]
\theoremstyle{plain} \newtheorem{lemmaa}[statement]{Lemma}
\theoremstyle{plain} 
\theoremstyle{definition} \newtheorem{deff}{Definition}
\theoremstyle{definition} 
\theoremstyle{remark} \newtheorem{rem}[statement]{Remark}

\DeclareMathOperator{\dist}{dist}

\begin{document}

\title[Generalizations of theorems of Maizel and Pliss]{Generalizations of analogs of theorems of Maizel and Pliss and their application in Shadowing Theory}
\author{Dmitry Todorov}
\address{Chebyshev laboratory, Saint Petersburg State University \\ 14th line of Vasiljevsky Island, 29B \\ 199178, Saint-Petersburg, Russia}
\thanks{Research supported by RFBR (project 12-01-00275) and the Chebyshev laboratory (grant of the Russian government N 11.G34.31.0026).}







\begin{abstract}
We generalize two classical results of Maizel and Pliss that describe relations between hyperbolicity properties of linear system of difference equations and its ability to have a bounded solution for every bounded inhomogeneity. We also apply one of this generalizations in shadowing theory of diffeomorphisms to prove that some sort of limit shadowing is equivalent to structural stability.
\end{abstract}


\maketitle

\section{Introduction}

In \cite{PERRON} Perron defined property (B) for systems of differential equations. The property is that an inhomogeneous system of differential equations has a bounded solution for every bounded inhomogeneity. ``Bounded'' here means that the standard $\sup$ norm on the space of continuous functions is bounded. In \cite{MAIZEL} A. Maizel proved a theorem that links property (B) on the half-line with the hyperbolicity property. In \cite{PLISSBSILSDE} Pliss characterized an analog of the property (B) on the full line in terms of hyperbolicity on two half-lines. The proof of the Pliss' theorem is based on the Maizel's theorem.

There are lots of papers devoted to the study of connections between Perron property (which is often called admissibility) and hyperbolicity (which has a form of exponential dichotomy in this sort of papers). For references, see \cite{BASK__SPEC_ANAL_DEFF_OPER_UNBOUND_COEF_09, LAT__FRED_PROP_EVOL_SEMIGROUPS_04, LAT__DICH_FRED_PROPS_EVOL_EQ_07, SASU__TRANS_INV_SPACES_ASYMPT_PROPS_VAR_EQ_11}.

The papers mentioned generalize statements, similar to Maizel and Pliss theorems, in different ways. There are two types of such results -- ones that widen the class of spaces to which theorem can be applied and ones that replaces the Perron property by some similar notion. We are interested in the first ones. All the recent results deal with infinite-dimensional case but we will only use their finite-dimensional versions.

Consider inhomogeneous system of linear nonautonomous equations
\begin{IEEEeqnarray}{rCll}
\dot{x} & = & A(t)x + f(t),\quad & t\in I. \label{eqs:inhomcont}
\end{IEEEeqnarray}
%
Here $I$ is either a half-line $[0,\infty)$ or a full line $\Rb$ and $A(t)$ is
\hl{a linear}
operator mapping $\Rb^d$ to itself for each $t.$ We assume that operators $A(t)$ have uniformly bounded norms $\normop{A(t)}.$
We say that a linear space $B,$ containing functions that map $I$ to $\Rb^d,$ is admissible for system \eqref{eqs:inhomcont} if for each inhomogeneity $f$ from $B$ there exists a solution $x$ that belongs to $B.$ Maizel in his original paper \cite{MAIZEL} considered the case when $I = [0,+\infty)$ and $B$ is the space of bounded continuous functions. He proved that the space of bounded continuous functions is admissible for system \eqref{eqs:inhomcont} if and only if this system posses exponential dichotomy.  Summarizing the results of \cite{LAT__FRED_PROP_EVOL_SEMIGROUPS_04, SASU__DISCR_ADMISS_EXP_DICH_EVOL_FAM_03, SASU__EXP_DICH_LPLQ_ADM_HALF_06, SASU__UNIFORM_DICH_EXP_DICH_HALFLINE_06, BASK__SPEC_ANAL_DEFF_OPER_UNBOUND_COEF_09}, Maizel theorem also holds when $B$ is one of the following spaces:
\begin{itemize}
\item space $L_p,\ p\in [1,\infty]$;
\item space $C_0$ of continuous functions, tending to zero at infinity;
\item space $C_{b,0}$ of bounded continuous functions, tending to zero at infinity;
\item homogeneous spaces that generalize \hl{the preceding three cases} (see definition in \cite{BASK__SPEC_ANAL_DEFF_OPER_UNBOUND_COEF_09} ).
\end{itemize}

Pliss considered the case when $I=\Zb$ and he proved that the space of bounded continuous functions is admissible for system \eqref{eqs:inhomcont} if and only if this system posses exponential dichotomy on both $[0,+\infty)$ and $(-\infty,0],$ and stable subspaces at zero are transverse for these two dichotomies.  As far as we discovered, there
\hl{does not}
exist any direct generalization of Pliss theorem but there exist some characterizations of invertibility and Fredholm properties of the operator $(L x)(t) = \dot{x}(t) - A(t)x(t),$ associated with system \eqref{eqs:inhomcont}, in terms of dichotomy of system \eqref{eqs:inhomcont} on two half-lines and certain properties of dichotomy projections at zero. See \cite{BASK__INV_FREDHOLM_PROPS_DIF_OPER_00, LAT__DICH_FRED_PROPS_EVOL_EQ_07} for exact statements. Authors in these papers deal with spaces $C_0$ and $L_p,\ p\in [1,\infty).$

One can see that all the spaces mentioned above have certain homogeneity properties. We prove generalizations of theorems of Maizel and Pliss for difference equations for the case of spaces of sequences with prescribed decay rate.
 Such spaces, in contrast, don't have homogeneity properties.


The discrete analog of the Pliss' theorem is widely used in shadowing theory (see \cite{PILTIKLSISS, TIKHHOL, PILMELISP}).
The theory of shadowing of approximate trajectories (pseudotrajectories)
of dynamical systems is now a well developed part of the global theory of
dynamical systems (see, for example, monographs \cite{PILSDS, PALSDS}). In particular the connections between shadowing and structural stability are of big interest.
There exist many types of shadowing for diffeomorphisms of closed manifolds. The simplest one is the standard shadowing, that is also often called a pseudo-orbit tracing property (POTP). Let $M$ be a closed Riemannian manifold.
Diffeomorphism $f:M\to M$ is said to have POTP if for a given accuracy any pseudotrajectory with errors small enough  can be approximated (shadowed) by
\hl{an}
exact trajectory.
It is shown in \cite{SAKAI__POTP_STC_DIFF_CLOSED_MAN} that an interior in $C^1$-topology of a set of diffeomorphisms of a manifold $M$ having standard shadowing property coincides with the set of structurally stable diffeomorphisms of $M.$

Standard shadowing doesn't provide a good control on the accuracy of shadowing in terms of a size of pseudotrajectory errors. Lipschitz shadowing property provides a linear control.
It is well known that uniformly hyperbolic systems satisfy not just POTP, but also Lipschitz shadowing property. The same was proved for structurally stable systems (see \cite{PILSDS}). Recently it was shown that Lipschitz shadowing is equivalent to structural stability (see \cite{PILTIKLSISS}).
It was also shown that structural stability is equivalent to H\"{o}lder shadowing property under certain additional assumptions (see \cite{TIKHHOL}).

One can also move in different direction and restrict a set of pseudotrajectories. In particular, one may ask pseudotrajectory $\seqz{x}{k}$ to have errors $\dist(f(x_k),x_{k+1})$ that tend to zero with $\vmod{k}\to\infty$ and the distance $\dist(y_k,x_k)$ between point of a shadowing trajectory $\seqz{y}{k}$ and point of the pseudotrajectory to tend to zero when $\vmod{k}\to\infty.$ This is called two-sided limit shadowing property.
It is known that structurally stable system has two-sided limit shadowing property but even the $C^1$-interior of the set of diffeomorphisms having two-sided shadowing property
\hl{does not coincide}
with the set of structurally stable diffeomorphisms (see \cite{PILLMSP}).

To make two-sided limit shadowing closer to structural stability it is enough to ask for linear control, summability of errors and summability of distances between points of the pseudotrajectory and points of shadowing trajectory. This is called Lipschitz $L_p$ shadowing property. The fact that Lipschitz $L_p$ shadowing follows from a uniform hyperbolicity of a system is well known (see \cite{PILSDS}). Recently, in \cite{FAK__DIFF_WITH_LP_SHAD_PROP} it is shown that the Lipschitz $L_p$ shadowing also follows form structural stability and that $C^1$ interior of a set of diffeomorphisms having Lipschitz $L_p$ shadowing property coincides with the set of structurally stable diffeomorphisms of $M.$

We introduce a property that is in some sense similar to Lipschitz $L_p$ shadowing, but instead of summability of errors and distances we require a polynomial decay rate. We show that this property is equivalent to structural stability. We can't use techniques, developed in \cite{PILTIKLSISS} directly, because they use discrete version of the original Pliss theorem as an essential part. Existing versions of this theorem are not applicable to the case of decaying sequences. To overcome this difficulty, we provide a generalization of this theorem, that suits our case.




\section{Definitions} \label{sec:defs}

Let $I$ be either $\Zplus = \setdef{k\in\Zb}{k\geq 0}$ or $\Zminus = \setdef{k\in\Zb}{k\leq 0}$ or $\Zb.$ Let $\Ac=\{A_k\}_{k\in I}$ be a sequence of linear isomorphisms $\Rb^d\to\Rb^d$ indexed by integers from $I.$ Consider homogeneous and inhomogeneous equations associated with this sequence.

\begin{IEEEeqnarray}{rCll}
x_{k+1}&=&A_k x_k,\qquad &k\in I; \label{eq:homogen} \\
x_{k+1}&=&A_k x_k + f_{k+1},\qquad &k\in I. \label{eq:nonhomogen}
\end{IEEEeqnarray}

\begin{rem}
For $I = \rayp$ we take $f_k$ to be defined for $k\geq 0$ and $f_0=0.$
\end{rem}

\noindent Define an analog of 
\hl{Cauchy matrix}
for equations \eqref{eq:homogen}:
$$
	\Phi_{m,l} = \begin{cases}
		A_{m-1}\circ\ldots\circ A_{l}, & m > l,\\
		Id, & m=l,\\
		A_{m}^{-1}\circ\ldots\circ A_{l-1}^{-1}, & m < l.
	\end{cases}
$$

Fix $\omega \geq 0.$ We use linear subspaces of the space of sequences of vectors from $\Rb^d,$ indexed by integers from $I.$ Denote the Banach space of sequences with bounded norm $\normban{x}_\omega=\sup\limits_{k\in I} \vmod{x_k} (|k|+1)^\omega$ by $\Nc_{\omega}(I).$ Such spaces have been already studied in a similar context (see \cite{BICHEG_COND_INVERT_DIF_OPER}). It is important to note that these spaces are neither homogeneous in the sense of Baskakov (see \cite{BASK__SPEC_ANAL_DEFF_OPER_UNBOUND_COEF_09}) nor
translation-invariant in the sense of Sasu (see \cite{SASU__TRANS_INV_SPACES_ASYMPT_PROPS_VAR_EQ_11}).

\begin{deff}
We say that a sequence $\Ac$ has Perron property $B_{\omega}(I)$ if for any sequence $f\in \Nc_\omega(I)$ there exists a solution of the inhomogeneous system of difference equations with inhomogeneity $f$ that belongs to $\Nc_\omega(I).$
\end{deff}

We use the following definition from \cite{PILGNH}:

\begin{deff} We say that a sequence $\mathcal{A}$ is hyperbolic on $I$ if there exist constants $K > 0,$ $\lambda\in (0,1)$ and projections $P_k, Q_k,\ k\in I$ such that if $S_k=P_k\Rb^d$ and $U_k=Q_k\Rb^d$ then the following holds:
\begin{IEEEeqnarray}{c}
\Rb^d=S_k \oplus  U_k; \\
A_kS_k=S_{k+1},\ A_kU_k=U_{k+1};\\
|\Phi_{k,l}v|\leq K\lambda^{k-l}|v|,\ v\in S_l,\ k\geq l; \label{hypdef:stable1}\\
|\Phi_{k,l}v|\leq K\lambda^{l-k}|v|,\ v\in U_l,\ k\leq l; \label{hypdef:unstable1}\\
\normop{P_k},\normop{Q_k} \leq K. \label{hypdef:projbound}
\end{IEEEeqnarray}
Everywhere here we mean that all indices are from $I.$
\end{deff}

\begin{rem}
We call the spaces $S_k$ and $U_k$ stable and unstable spaces of the sequence $\Ac.$
\end{rem}

\begin{rem} \label{rem:projbound}
If norms of all $\normop{A_k}$ and $\normop{A_k}^{-1}$ are bounded then conditions \eqref{hypdef:stable1} and \eqref{hypdef:unstable1} imply condition \eqref{hypdef:projbound} (with different constant $K$ in general), which is equivalent to the boundedness from zero of the angle between the spaces $S_k$ and $U_k$ (see \cite{DALKREINSTAB} p. 224, 234, 237 for example).
\end{rem}

\begin{rem} \label{rem:hyperbound_omega}
Let $I=\Zplus.$ Then conditions \eqref{hypdef:stable1} and \eqref{hypdef:unstable1} for some $\lambda\in (0,1)$ and $K>0$ follow from the existence of $\lambda_1 \in (0,1)$ and $K_1 > 0$ such that the following estimates hold
\begin{IEEEeqnarray}{c}
|\Phi_{k,l}v|\leq K_1\lambda_1^{k-l}(k+1)^{-\omega} (l+1)^{\omega}|v|,\ v\in S_l,\ k\geq l; \label{hypdef:stable2} \\
|\Phi_{k,l}v|\leq K_1\lambda_1^{l-k}(k+1)^{-\omega} (l+1)^{\omega}|v|,\ v\in U_l,\ k\leq l. \label{hypdef:unstable2}
\end{IEEEeqnarray}
Also everywhere here we mean that all indices are from $I.$
\end{rem}

\begin{proof}
Let conditions \eqref{hypdef:stable2} and \eqref{hypdef:unstable2} be satisfied. We show that conditions \eqref{hypdef:stable1} and \eqref{hypdef:unstable1} are also satisfied:
$$ K_1 \lambda_1^{l-k} (k+1)^{-\omega} (l+1)^{\omega} \leq K_1 \lambda_1^{\frac{l}{2}} (l+1)^{\omega} \leq $$
$$ \leq \enbrace{\max_{l\in\Zplus} \enbrace{K_1 \lambda_1^{\frac{l}{2}} (l+1)^{\omega}} } \enbrace{\lambda_1^{ \frac{1}{2} } }^{l} = K_2\lambda_2^l \leq K_2 \lambda_2^{l-k},\ 2k\leq l, $$
where $\lambda_2 = \sqrt{\lambda_1}$ and $K_2 = \enbrace{K_1 \lambda_1^{\frac{l}{2}} (l+1)^{\omega}};$ and
$$ K_1 \lambda_1^{l-k} (k+1)^{-\omega} (l+1)^{\omega} \leq K_1 \lambda_1^{l-k} (k+1)^{-\omega} (2k+1)^{\omega} = $$
$$ = 2^{\omega} K_1 \lambda_1^{l-k} \enbrace{\frac{k+\frac{1}{2}}{k+1}}^{\omega} \leq K_3 \lambda_1^{l-k},\ 2k \geq l \geq k,$$
where $K_3 = 2^{\omega} K_1 \enbrace{\frac{k+\frac{1}{2}}{k+1}}^{\omega}.$
This proves inequality \eqref{hypdef:unstable1} for $\lambda = \max(\lambda_1,\lambda_2)$ and $K=\max(K_1,K_2,K_3).$ Inequality \eqref{hypdef:stable1} is obvious since
$$(k+1)^{-\omega} (l+1)^{\omega} \leq 1,\ k\geq l.$$

\end{proof}

\section{Main Results}

We prove the following theorem in Section \ref{sec:maizelthm}:

\begin{theorem} [a generalization of the discrete analog of Maizel Theorem] \label{thm:Maizelcor}
Let $I=\Zplus$ and the norms of all matrices $A_k$ and $\inv{A_k}$ be bounded by $M>0.$ A sequence $\Ac$ has property $B_\omega(I)$ iff it is hyperbolic on  $\Zplus.$
\end{theorem}

We prove the following theorem in Section \ref{sec:plissthm}:

\begin{theorem} [a generalization of the discrete analog of Pliss Theorem] \label{thm:mainpliscor}
Let $I=\Zb$ and the norms of all matrices $A_k$ and $\inv{A_k}$ be bounded by $M>0.$ A sequence $\Ac$ has property $B_\omega(I)$ iff it is hyperbolic on both  $\Zplus$ and $\Zminus$ and the spaces $B^+(\Ac)$ and $B^-(\Ac)$ are transverse. Here
\begin{eqnarray*}
B^+(\Ac)=\setdef{v\in\Rb^d}{\vmod{\Phi_{k,0}v}\to 0,\ k\to +\infty}, \\
B^-(\Ac)=\setdef{v\in\Rb^d}{\vmod{\Phi_{k,0}v}\to 0,\ k\to -\infty}.
\end{eqnarray*}
\end{theorem} 

\section{Generalizations of discrete analogs of theorems of Maizel and Pliss} \label{sec:discreteth}

We prove generalizations of theorems of Maizel and Pliss for the case of difference equations. 

\subsection{Maizel Theorem} \label{sec:maizelthm}

Let $I=\Zplus.$ For brevity we write $\Nc_\omega$ instead of $\Nc_\omega(I).$  Assume that the sequence $\Ac$ has property $B_{\omega}(I)$
\hl{ and choose the number $M$ that bounds norms of all $A_k$ and } $A^{-1}_k$
\hl{ large enough. The last means that, in particular, $2M>1.$ }

\medbreak

Denote
$$V_1 = \left\{ x_0\left| x\in \Nc_\omega,\ x \hbox{ is a solution of homogeneous equation \eqref{eq:homogen}} \right. \right\}.$$
Since equation \eqref{eq:homogen} is linear and $\Nc_\omega$ is a linear space, $V_1$ is also a linear space. Denote the orthogonal complement of $V_1$ by $V_2$ and orthogonal projection onto $V_1$ by $P.$

It is easy to see that the following holds:
\begin{statement} \label{statement:zeroel}
For any sequence $f\in \Nc_\omega$ there exists
\hl{exactly}
one solution $T(f)\in \Nc_\omega$ of inhomogeneous equation \eqref{eq:nonhomogen} with inhomogeneity $f$ such that $(T(f))_0\in V_2.$
\end{statement}

\begin{statement}  \label{statement:contoper}
For any sequence $\Ac$ the operator $T:\Nc_\omega \to \Nc_\omega$ from the previous statement is continuous. In particular there exists a positive $r$ such that
\begin{equation*} \normban{Tf}_\omega\leq r \normban{f}_\omega. \end{equation*}
\end{statement}

\begin{proof}
Fully analogous to the proof of Statement 4 from \cite{COPDFA}.
\end{proof}
\medskip

From now on we use the operator $T$ and number $r$ from the previous statement. Also we suppose that $r\geq 1$ and that the number $M$ from the statement of Theorem \ref{thm:Maizelcor} satisfies inequality $rM\geq1.$


\subsubsection{Technical lemmas}

Denote
$$
X_k = \begin{cases}
		\Phi_{k,0}, & k > 0,\\
		Id, & k=0,\\
		\Phi_{0,-k}, & k<0.\\
	\end{cases}
$$

It is easy to see that the following holds.
\begin{statement}
The following formula provides a solution of inhomogeneous difference equation \eqref{eq:nonhomogen}:
\begin{equation}
y_k=\sum^{k}_{u=0} \sfm{k}{u} f_u - \sum^{\infty}_{u=k+1} \ufm{k}{u}f_{u}, \label{eq:halfexactsol}
\end{equation}
when the series in the second summand converges. Here we take $f_0$ equal to $0.$
\end{statement}


\begin{rem}
The following formula can be interpreted as an analog of the Green's function for difference equations
$$
G_{k,u}=\begin{cases} \sfm{k}{u}, &0\leq u \leq k, \\
-\ufm{k}{u}, &0\leq k < u. \end{cases}
$$
So formula \eqref{eq:halfexactsol} can be rewritten in a more compact way:
\begin{equation}
y_k=\sum_{u=0}^{\infty} G_{k,u}f_{u}.
\label{eq:nonhomsolexplicithalf}
\end{equation}
\end{rem}
\medskip

\begin{lemmaa}
Let $k_0, k_1, k$ be nonnegative integers and $\xi\in \Rb^d$ be a nonzero vector.

Then the following inequalities hold
\begin{IEEEeqnarray}{rCll}
\vmod{X_kP\xi}\sumd{k_0}{k} (u+1)^{-\omega} \vmod{X_{u}\xi}^{-1} &\leq  &r(k+1)^{-\omega},\quad & 0\leq k_0\leq k,
\label{neq:tl1np}\\
\vmod{X_k(I-P)\xi}\sumd{k}{k_1} (u+1)^{-\omega} \vmod{X_{u}\xi}^{-1} &\leq & 2rM (k+1)^{-\omega},\quad & 0\leq k \leq k_1.
\label{neq:tl1nip}
\end{IEEEeqnarray}

\end{lemmaa}
\begin{proof}
Fix nonnegative integers $l_0, l_1$ such that $l_0\leq l_1.$
Consider a sequence $f$ with $f_i = 0,\ i > l_1.$ Then formula \eqref{eq:nonhomsolexplicithalf} looks like:
$$y_l=\sumd{0}{l_1} G_{l,u} f_u.$$
For $l\geq l_1$ all the indices $u$ in the sum are less or equal than $l_1$ and the first string from the definition of $G_{l,u}$ is used. The previous equality turn into the following:
$$y_l=X_l P\sumd{0}{l_1} X_{-u} f_u.$$
Thus the vector $y_l$ for $l\geq l_1$ is an image of the vector from $V_1$ that is independent of $l.$ This means that all the sequence $y$ except a finite number of entries is a solution of homogeneous equation \eqref{eq:homogen} with initial conditions from $V_1$. Thus $y$ belongs to $\Nc_\omega.$ Using that $f_0=0$ we obtain
\begin{equation*}
y_0=-(I-P)\sumd{0}{l_1} X_{-u} f_u \in V_2.
\end{equation*}
So $y=Tf$ and therefore $\normban{y}_\omega\leq r\normban{f}_\omega.$
Let $x_i=X_i\xi.$ We define the sequence $f:$
$$
f_i= \begin{cases} 0, & i < l_0,
\\ (i+1)^{-\omega}\displaystyle\frac{x_i}{\vmod{x_i}}, & l_0\leq i\leq l_1,
\\ 0, & i > l_1. \end{cases}
$$
Then $\normban{f}_\omega = 1.$ Substituting the formula for a solution in the inequality from Statement \ref{statement:contoper} we obtain
\begin{equation}
\vmod{ \sumd{l_0}{l_1} (u+1)^{-\omega} G_{l,u} \frac{x_u}{|x_u|} } = \vmod{y_l} \leq r (l+1)^{-\omega}. \label{neq:Gk0k1}
\end{equation}
For $l_1=l=k,\ l_0 = k_0$ from \eqref{neq:Gk0k1} we obtain that if $k\geq k_0$ then
$$
r (k+1)^{-\omega} \geq   \vmod{ \sumd{k_0}{k} (u+1)^{-\omega} G_{k,u} \frac{x_u}{|x_u|} } = \vmod{\sumd{k_0}{k} (u+1)^{-\omega} \sfm{k}{u} \frac{X_u\xi}{\vmod{X_u\xi}}}=
$$
$$
=\vmod{X_kP\xi\sumd{k_0}{k} (u+1)^{-\omega}\vmod{X_{u}\xi}^{-1}} =\vmod{X_kP\xi}\sumd{k_0}{k} (u+1)^{-\omega}\vmod{X_{u}\xi}^{-1}.
$$
To prove the second inequality from the statement of the lemma we do the similar. The important thing to notion here is that in the second string of the definition of $G_{k,s}$ the inequality is strict. Then for $l=k-1,\ l_0=k,\ l_1=k_1$ from \eqref{neq:Gk0k1} we obtain that for $0<k\leq k_1$ we have
$$
r k^{-\omega} \geq \vmod{ \sumd{k}{k_1} (u+1)^{-\omega} G_{k-1,u} \frac{x_u}{|x_u|} } = \vmod{\sumd{k}{k_1} -(u+1)^{-\omega}\ufm{k-1}{u}\frac{X_{u}\xi}{\vmod{X_{u}\xi}}}=
$$
$$
=\vmod{X_{k-1}(I-P)\xi\sumd{k}{k_1} (u+1)^{-\omega}\vmod{X_{u}\xi}^{-1}}
=\vmod{X_{k-1}(I-P)\xi}\sumd{k}{k_1} (u+1)^{-\omega}\vmod{X_{u}\xi}^{-1}=
$$
$$
=\vmod{A_{k-1}^{-1}X_{k}(I-P)\xi}\sumd{k}{k_1} (u+1)^{-\omega}\vmod{X_{u}\xi}^{-1} \geq $$
$$\geq \inv{\normop{A_{k-1}}} \vmod{X_{k}(I-P)\xi}\sumd{k}{k_1} (u+1)^{-\omega}\vmod{X_{u}\xi}^{-1}. $$
Now we prove the second inequality of the statement of the lemma for the case when $0=k < k_1$ using the previous inequality for $k=1:$
$$
\vmod{X_{0}(I-P)\xi}\sumd{0}{k_1} (u+1)^{-\omega}\vmod{X_{u}\xi}^{-1} = \vmod{X_{0}(I-P)\xi}\sumd{1}{k_1} (u+1)^{-\omega}\vmod{X_{u}\xi}^{-1} + \vmod{(I-P)\xi}\leq
$$
$$
\leq \normop{\inv{A_0}}\vmod{X_{1}(I-P)\xi}\sumd{1}{k_1} (u+1)^{-\omega}\vmod{X_{u}\xi}^{-1} + 1 \leq
r\normop{\inv{A_0}}+1.
$$
For $k=k_1=0$ the inequality is obvious.
\end{proof}

\medskip


\begin{lemmaa}
Let $k_0, k_1,$ $k, s$ be nonnegative integers and $\xi$ be a vector.

Denote
$$\mu = {1-\inv{(4rM)} }.$$

\noindent The following inequalities are satisfied:

if $P\xi\neq 0$ then
\begin{IEEEeqnarray}{c}
\sumd{k_0}{s} (u+1)^{-\omega} \vmod{X_uP\xi}^{-1} \leq  \mu^{k-s} \sumd{k_0}{k} (u+1)^{-\omega} \inv{\vmod{X_uP\xi}} \label{neq:lgnp}
\end{IEEEeqnarray}
for  $k  \geq s \geq k_0;$

if $(I-P)\xi\neq 0$ then
\begin{IEEEeqnarray}{c}
\sumd{s}{k_1} (u+1)^{-\omega} \vmod{X_u(I-P)\xi}^{-1} \leq \mu^{s-k} \sumd{k}{k_1} (u+1)^{-\omega} \vmod{X_u(I-P)\xi}^{-1} \label{neq:lgnip}
\end{IEEEeqnarray}
for  $k_1\geq s\geq k;$

\end{lemmaa}
\begin{proof}
\noindent Denote
$$\phi_i=\sumd{k_0}{i} (u+1)^{-\omega} \vmod{X_uP\xi}^{-1},\quad i\geq k_0.$$
$$\psi_i=\sumd{i}{k_1} (u+1)^{-\omega} \vmod{X_u(I-P)\xi}^{-1},\quad i\leq k_1,$$

We prove inequality \eqref{neq:lgnp}. Since $P\xi\neq 0,$ it is easy to see that $\phi_k>0.$

Also it is obvious that $\phi_k-\phi_{k-1}=(k+1)^{-\omega}\vmod{X_kP\xi}^{-1}.$ Thus replacing $\xi$ by $P\xi$ in \eqref{neq:tl1np} we get
$$
\frac{\phi_k}{\phi_k-\phi_{k-1}}\leq r\leq 2rM.
$$
Then
$$
(2rM)^{-1} \leq \frac{\phi_k-\phi_{k-1}}{\phi_k} = 1-\frac{\phi_{k-1}}{\phi_k}.
$$
Therefore $\phi_{k-1} \leq {\enbrace{1-(2rM)^{-1}}}\phi_k.$ If we consequently use this inequality enough times, we obtain
$$
\phi_s\leq \enbrace{1-(2rM)^{-1}} \dots \enbrace{1-(2rM)^{-1}} \phi_k = $$
$$= \enbrace{1-(2rM)^{-1}}^{k-s} \phi_k,\quad k \geq s.$$

To prove the second inequality from the statement of the lemma recall that $\psi_k>0.$ Analogous to the proof of the first inequality 
\hl{from} 
the statement of the lemma, $\psi_k-\psi_{k+1}=(k+1)^{-\omega}|X_k(I-P)\xi|^{-1}$ and replacing $\xi$ by $(I-P)\xi$ in \eqref{neq:tl1nip} we have
$$ \psi_{k+1} \leq \enbrace{1-\inv{(2rM)}k^{\omega}(k+1)^{-\omega}} \psi_k \leq $$
$$\leq \enbrace{1-\inv{(4rM)} } \psi_k. $$
Again, if we consequently use this inequality enough times, we get inequality \eqref{neq:lgnip}.
\end{proof} 
\subsubsection{Proof of the discrete analog of the Maizel theorem} \label{ssec:maizelproof}

\begin{theorem}\label{thm:rawMaizel}
The following inequalities holds
$$
\normop{\sfm{k}{s}}\leq r^2 (k+1)^{-\omega}(s+1)^{\omega} \mu^{k-s}
$$
for $0\leq s\leq k;$

$$
\normop{\ufm{k}{s}}\leq 2 r^2 M^2 (k+1)^{-\omega}(s+1)^{\omega} \mu^{s-k}
$$
for $0\leq k < s.$
\end{theorem}

\begin{proof}
Fix a natural 
\hl{number} 
$s\geq 1$ and a unit vector $\xi.$ Define a sequence $y:$
$$
y_k=\begin{cases} -\ufm{k}{s}\xi, & 0\leq k < s,\\
\sfm{k}{s}\xi, & k\geq s. \end{cases}
$$
The sequence $y$ coincides (except a finite number of entries) with a solution of homogenous equation \eqref{eq:homogen} with initial conditions from $V_1$ and therefore $y$ belongs to $\Nc_\omega.$ Now we define a sequence $f$ in such a way 
\hl{that}
the sequence $y$ is a solution of inhomogeneous equation \eqref{eq:nonhomogen} with inhomogeneity  $f:$
$$
f_k=\begin{cases} 0, & k\neq s,\\
\xi, & k=s. \end{cases}
$$
It is easy to see that in this case $y$ becomes a solution. This means that $y=Tf.$ Thus $\normban{y}_\omega\leq r\normban{f}_\omega = r(s+1)^{\omega}.$ We prove the first inequality for the operator norms from the statement of the theorem. Using the definition of the sequence $y$ we can write
$$ \vmod{ \sfm{k}{s} \xi } = \vmod{y_k} \leq r(k+1)^{-\omega}(s+1)^\omega, \quad s\leq k. $$
Since $\xi$ can be any unit vector, this gives us an estimate for the operator norms of $\sfm{k}{s}.$ Now we can replace 
\hl{$\xi$ }
by the solution of the homogeneous equation $x_k=X_k\xi$ and substitute in the previous inequality instead of $\xi:$
\begin{equation}
\vmod{ X_k P\xi } =\vmod{\sfm{k}{s}x_s}\leq r(k+1)^{-\omega}(s+1)^\omega\vmod{x_s}, \quad s\leq k.  \label{neq:seqel}
\end{equation}
Let $P\xi\neq 0.$ Consequently using inequalities \eqref{neq:tl1np} and \eqref{neq:lgnp} for $k_0=s$ and \eqref{neq:seqel} for $k=s$ we get
$$
\vmod{\sfm{k}{s}x_s}=\vmod{ X_k P\xi } \leq
 r(k+1)^{-\omega}\enbrace{\sum_{u=s}^k (u+1)^{-\omega} \vmod{ X_u P\xi }^{-1}}^{-1}\leq$$
$$\leq r(k+1)^{-\omega}\inv{\enbrace{ \mu^{-(k-s)} (s+1)^{-\omega} \inv{\vmod{X_sP\xi}} }}= $$
$$ = r(k+1)^{-\omega} \mu^{k-s} (s+1)^{\omega} \vmod{X_sP\xi}\leq$$
$$\leq r^2 (k+1)^{-\omega}(s+1)^{\omega} \mu^{k-s}\vmod{x_s}.$$
If $P\xi = 0$ then the resulting inequality is obvious. Since $x_s=X_s\xi$ and $X_s$ is an isomorphism, we have an estimate for the operator norm:
$$
\normop{\sfm{k}{s}}\leq r^2 (k+1)^{-\omega}(s+1)^{\omega} \mu^{k-s},\quad 1\leq s\leq k.
$$
In this reasoning we have used only the fact that inequality \eqref{neq:seqel} is satisfied for $s=k.$ This is also true for $s=k=0$ since $\normop{P}\leq 1.$ Therefore, we proved the estimate for $0\leq s\leq k.$

The proof of the second estimate from the statement for $s>k$ is similar to the proof of the first estimate. The only small differences are due to the fact that now we cannot use an analog of inequality \eqref{neq:seqel} for $k=s$ because in the definition of the sequence $y$ the numbers $s$ and $k$ cannot be equal. The following inequality can be proved in a very same manner as one in the proof of the first estimate:
\begin{equation*}
\vmod{ X_k (I-P)\xi } =\vmod{\ufm{k}{s}x_s}\leq r(k+1)^{-\omega}(s+1)^\omega\vmod{x_s}, \quad s > k.  
\end{equation*}
For $k=s-1$ we multiply the vector inside the norm brackets by $A_{s-1}:$
$$
\vmod{X_s(I-P)\xi}=\vmod{A_{s-1}\ufm{s-1}{s}x_{s}} \leq
$$
\begin{equation}
\leq \normop{A_{s-1}}\vmod{\ufm{s-1}{s}x_{s}} \leq r M (s+1)^\omega s^{-\omega} \vmod{x_s}. \label{neq:seqel2}
\end{equation}
After that the proof is fully analogous to the proof of the first estimate.
\end{proof}

\begin{lemmaa} \label{lemma:seriesbound}
For any $\lambda\in (0,1)$ there exists a constant $C>0$ depending only on $\lambda$ and $\omega$ such that
\begin{equation} \label{neq:seriessol}
(k+1)^{\omega} \enbrace{ \sum_{u=0}^k \lambda^{k-u} (u+1)^{-\omega} + \sum_{u=k+1}^{\infty} \lambda^{u-k} (u+1)^{-\omega}}  < C
\end{equation}
for any $k\geq 0.$
\end{lemmaa}

\begin{proof}
We estimate first summand from \eqref{neq:seriessol}. To do this it is enough to estimate the corresponding integral:
$$
(k+1)^{\omega} \int_0^k \lambda^{k-u} (u+1)^{-\omega} du =
$$
$$ = (k+1)^{\omega} \enbrace{ \int_0^{\half{k}} \lambda^{k-u} (u+1)^{-\omega} du + \int_{\half{k}}^k \lambda^{k-u} (u+1)^{-\omega} du }.
$$
Now we estimate separately the two integrals from the previous formula. The first one can be estimated in the following way:
$$
(k+1)^{\omega} \int_0^{\half{k}} \lambda^{k-u} (u+1)^{-\omega} du \leq
(k+1)^{\omega} \int_0^{\half{k}} \lambda^{k-u} du \leq
$$
$$
\leq (k+1)^{\omega} \int_0^{\half{k}} \lambda^{\half{k}} du =
\half{1} k(k+1)^{\omega} \lambda^{\half{k} } \leq C_0.
$$
The second one can be estimated in the following way:
$$
(k+1)^{\omega} \int_{\half{k}}^k \lambda^{k-u} (u+1)^{-\omega} du \leq
\int_{\half{k}}^k \lambda^{k-u} \enbrace{\frac{k+1}{\half{k}+1}}^{-\omega} du \leq
$$
$$
\leq C_1 \lambda^k \int_{\half{k}}^{k} \lambda^{-u} du = C_2 \lambda^k \enbrace{\lambda^{-k} - \lambda^{-\half{k}}} = C_2 \enbrace{1 - \lambda^{\half{k}}} \leq C_3.
$$

Here is the estimate for the second summand from \eqref{neq:seriessol}:
$$
(k+1)^{\omega} \sum_{u=k+1}^{\infty} \lambda^{u-k} (u+1)^{-\omega} \leq \sum_{u=k+1}^{\infty} \lambda^{u-k} < C_4.
$$

\end{proof}

Now we prove Theorem \ref{thm:Maizelcor}. We show how property $B_\omega(\Zplus)$ implies hyperbolicity.
Let
$$ K=2r^2M^2,\ \lambda = (1-(4rM)^{-1}),\ P_l=X_lPX_{-l},\ Q_l=X_l(I-P)X_{-l}.$$
Using Theorem \ref{thm:rawMaizel} it is easy to check that the first two conditions from the definition of hyperbolicity and inequalities from Remark \ref{rem:hyperbound_omega} are satisfied. The uniform estimates of the norms of the projectors $P_k$ and $Q_k$ are due to Remark \ref{rem:projbound}.

Now we show how hyperbolicity implies $B_\omega(\Zplus).$ Let $\normban{f}_{\omega} < R.$ We define a sequence $y_k$ as follows:
$$
y_k=\sum^{k}_{u=0} \Phi_{k,u} P_u f_u - \sum^{\infty}_{u=k+1} \Phi_{k,u} Q_u f_{u}.
$$
Let $K,\lambda$ be the numbers from the definition of hyperbolicity of the sequence $\Ac.$ Then using the Lemma \ref{lemma:seriesbound} we write this estimates:
$$
(k+1)^{\omega}\vmod{y_k} \leq R (k+1)^{\omega} \enbrace{\sum^{k}_{u=0} (u+1)^{-\omega} \normop{\Phi_{k,u} P_u} + \sum^{\infty}_{u=k+1} (u+1)^{-\omega} \normop{\Phi_{k,u} Q_u} } \leq$$
$$\leq RK (k+1)^{\omega} \enbrace{\sum^{k}_{u=0} \lambda^{k-u} (u+1)^{-\omega}  + \sum^{\infty}_{u=k+1} \lambda^{u-k} (u+1)^{-\omega} } < C.
$$

\subsection{Pliss Theorem} \label{sec:plissthm}

Let $I=\Zb$ and $\omega\geq 0.$ We assume that the norms of $A_k$ and $\inv{A_k}$ are bounded.

\begin{statement} \label{statement:hyperZ}
If a sequence $\Ac$ have property $B_\omega (\Zb)$ then it is hyperbolic on $\rayp$ and $\rayn$ with corresponding stable and unstable spaces $S_k^+,\ U_k^+$ and $S_k^-,\ U_k^-.$
\end{statement}
\begin{proof}
Since we have property $B_\omega(\Zb)$ for the sequence $\Ac,$ we also have properties $B_\omega(\rayp)$ for its positive part $\{A_k\}_{k=0}^{\infty}$ and $B_\omega(\rayn)$ for its negative part  $\{A_k\}_{k=-\infty}^{0}.$ Then the hyperbolicity on $\rayp$ follows directly from the Maizel theorem. In particular this means that there exist stable and unstable subspaces $S^+_k$ and $U^+_k.$ Hyperbolicity on  $\rayn$ also follows from Maizel theorem but it should be applied not to equations \eqref{eq:homogen} and \eqref{eq:nonhomogen} but to the equations with inverted time:
\begin{IEEEeqnarray}{rCll}
x_{k}&=&{A_k^{-1}} x_{k+1},\qquad &k\in\rayp, \nonumber \\
x_{k+1}&=&A_k^{-1} x_{k+1} - {A_k^{-1}}f_{k+1},\qquad &k\in\rayp \nonumber.
\end{IEEEeqnarray}
Thus hyperbolic sequence $\{A_{-k}^{-1}\}_{k=0}^\infty$ has spaces $\widetilde{S}_k$ and $\widetilde{U}_k.$ We denote $U_k^-=\widetilde{S}_k$ and $S_k^{-}=\widetilde{U}_k$ keeping in mind the sequence $\{A_k\}_{k=-\infty}^{0}.$

\end{proof}


\begin{statement} \label{statement:plissb2hyp}
If a sequence $\Ac$ have property $B_\omega(\Zb)$ then it is hyperbolic both on $\rayp$ and $\rayn$ and spaces $S^+_0$ and $U^-_0$ from Statement \ref{statement:hyperZ} are transverse.
\end{statement}

\begin{proof}

Let $S^+_0$ and $U^-_0$ be 
\hl{nontransverse.} 
Then there exists a vector $x$ such that $x\neq y_1 + y_2,$ where $y_1\in S^+_0,\ y_2\in U^-_0.$
We know that $S^+_0\bigcap U^+_0 = 0,\ S^+_0 + U^+_0 = \Rb^n$ therefore $x$ can be represented as  $x=\zeta+\eta$ with $\zeta\in S^+_0$ and $\eta\in U^+_0.$ Thus $\eta\neq z_1 + z_2$ for $z_1\in S^{+}_0,$ $z_2\in U^-_0.$ Take a sequence of numbers $a_k$ whose entries equal $0$ for negative indices and are in $(0,1)$ for nonnegative. We construct a sequence $\theta_k$ that will 
\hl{lead} 
us to a contradiction:

$$\theta_k=-\sum_{i=k+1}^\infty \Phi_{k,i}f_{i},\ k\in\Zb,$$
where

$$
f_i= a_i (i+1)^{-\omega} \frac{\Phi_{i,0}\eta}{|\Phi_{i,0}\eta|}.
$$

Vectors $f_i$ belong to $U_{i}^+,\ i\geq 0,$ because $\Phi_{i,j}$ maps $U^+_j$ to $U^+_i$
by the hyperbolicity definition. The series from the definition of $\theta_k$ converges:
$$
\left|\sum_{i=k+1}^\infty \Phi_{k,i}f_{i}\right|\leq\sum_{i=\max(0,k)+1}^\infty |\Phi_{k,i}f_{i}|
\leq \sum_{i=\max(0,k)+1}^\infty C|\eta|\lambda^{i-k} (i+1)^{-\omega} \leq
$$

$$
\leq C\vmod{\eta}\sum_{l=1}^\infty \lambda^{l} \leq C_1
$$

Recall that the sequence $\{ \theta_k \}_{k\in\rayp}$ belongs to $\Nc_\omega(\rayp):$
$$
(k+1)^{\omega} \left|\sum_{i=k}^\infty \Phi_{k,i}f_{i}\right|
\leq  (k+1)^{\omega}\sum_{i=\max(0,k)+1}^\infty |\Phi_{k,i}f_{i}| \leq
$$

$$
\leq (k+1)^{\omega}\sum_{i=\max(0,k)+1}^\infty (i+1)^{-\omega} \normop{\Phi_{k,i} \arrowvert_{U_{i}^+} } \vmod{\eta}
\leq C_2 \vmod{\eta} \sum_{i=\max(0,k)+1}^\infty \lambda^{i-k} \leq
$$

$$
\leq C_2\vmod{\eta}\sum_{l=1}^\infty \lambda_1^{l} \leq C_3.
$$

It is easy to see that the sequence $\theta_k$ is a solution of the inhomogeneous equation  \eqref{eq:nonhomogen}.
Moreover, the following equality is satisfied
\begin{equation*}
\theta_0 = -\sum_{i=1}^\infty \Phi_{0,i}f_{i}=C_4\eta. 
\end{equation*}

This means that for $\theta_0$ the same thing as for $\eta$ is true
\begin{equation} \label{neq:thetaneq}
\theta_0\neq y_1 + y_2,
\end{equation}
 with $y_1\in S^+_0,\ y_2\in U^-_0.$


Due to the property $B_\omega (\Zb)$ there exists a solution $\seqz{\psi}{k}\in\Nc_\omega(\Zb)$ of inhomogeneous equation \eqref{eq:nonhomogen} with inhomogeneity $f.$ Since all entries of $f$ with negative indices are equal to zero, the nonpositive part of $\seqz{\psi}{k}$ is a solution of the homogeneous equation \hl{ and therefore its entry at zero should belong to $U^-_0.$ }

Every two solutions of the inhomogeneous equation differ by a solution of the homogeneous equation.
Thus $\seq{\psi_k}_{k\in\Zb}=\seq{ X_k(\psi_0-\theta_0)+\theta_k }_{k\in\Zb}.$
So we have $\seq{ X_k(\psi_0-\theta_0) }_{k\in\rayp} = \seq{ \psi_k }_{k\in\rayp} - \seq{ \theta_k }_{k\in\rayp} \in \Nc_\omega(\rayp)$ since $\Nc_\omega(\rayp)$ is a linear space.
From this we obtain that the vector $\psi_0-\theta_0$ belongs to $S^+_0.$
Therefore if we denote $y_1=\theta_0-\psi_0,\ y_2=\psi_0$ then we have $\theta_0=y_1+y_2,$ 
\hl{which} 
contradicts inequality \eqref{neq:thetaneq}.
\end{proof}

Now we prove Theorem \ref{thm:mainpliscor}.

At first we show how the existence of property $B_\omega(\Zb)$ follows from the hyperbolicity on $\rayp$ and $\rayn$ and transversality of $B^+(\Ac)$ and $B^-(\Ac).$ Fix a sequence $f\in \Nc_\omega(\Zb).$ Consider its positive and negative parts
$$ f^+ = \seq{f_k}_{k\in\rayp},\quad f^- = \seq{f_k}_{k\in\rayn}. $$
Since the sequence $\Ac$ is hyperbolic on both $\rayp$ and $\rayn,$ by the Maizel theorem its positive and negative parts $\Ac^+ = \seq{A_k}_{k\in\rayp}$ and $\Ac^- = \seq{A_k}_{k\in\rayn}$ have properties   $B_\omega(\rayp)$ and $B_\omega(\rayn)$ correspondingly. Thus there exist solutions $\psi^+\in\Nc_\omega (\rayp)$ and $\psi^-\in\Nc_\omega (\rayn)$ of equations \eqref{eq:nonhomogen} for $I=\rayp$ and $I=\rayn$ with inhomogeneities $f^+$ and $f^-$ correspondingly. If $\psi^+_0 = \psi^-_0$ then the sequence $\psi$ with
$$\psi_k = \psi^-_k,\ k\leq 0,\qquad \psi_k = \psi^+_k,\ k > 0$$
is a solution of the inhomogeneous system \eqref{eq:nonhomogen} for $I=\Zb$ and belongs to $\Nc_\omega(\Zb).$
If $\psi^+_0 \neq \psi^-_0$ then the solutions $\psi^+$ and $\psi^-$ can be modified by solutions of the homogeneous systems: we show that there exist solutions $\phi^+\in\Nc_\omega (\rayp)$ and $\phi^-\in\Nc_\omega (\rayn)$ of the homogeneous system \eqref{eq:homogen} for $I=\rayp$ and $I=\rayn$ such that $\psi_0^+ + \phi_0^+ = \psi_0^- + \phi_0^-.$
The last condition can be rewritten as
\begin{equation} \label{eq:zeroelagree}
\psi_0^+ - \psi_0^- = \phi_0^- - \phi_0^+.
\end{equation}

Recall that $B^+(\Ac)=S^+_0$ and $B^-(\Ac)=U^-_0$ since for a hyperbolic sequence solutions of the corresponding homogeneous linear system of difference equations are either tend to infinity with exponential speed or tend to zero with exponential speed.

By assumption the spaces $B^+(\Ac)$ and $B^-(\Ac)$ are transverse so every vector from $\Rb^d$ can be represented as a difference from the right hand sid of \eqref{eq:zeroelagree}. In particular we can obtain the left hand side of \eqref{eq:zeroelagree}.

To obtain hyperbolicity and transversality from property $B_{\gamma}$ we only need to use Statement \ref{statement:plissb2hyp} and the fact that $B^+(\Ac)=S^+_0$ and $B^-(\Ac)=U^-_0.$

\section{Application of The Generalization of Discrete Analog of Pliss Theorem} \label{sec:sublinsh}

In this section we apply the generalized version of Pliss theorem for difference equations in shadowing theory.

\subsection{Definitions}

Let $f$ be a homeomorphism of a metric space $(M,\dist)$ and consider a dynamical system that is generated by $f.$

\begin{deff}
We say that a sequence $\seqz{x}{k}$ of points of $M$ is a $d$-pseudotrajectory of the dynamical system $f$ if the following inequalities are satisfied
\begin{equation*}
\dist(x_{k+1},f(x_k))<d,\quad k\in\Zb.
\end{equation*}
\end{deff}

Let $\gamma$ be a nonnegative real number.

\begin{deff}
We say that a sequence $\seqz{x}{k}$ of points of $M$ is a $\gamma$-decreasing $d$-pseudotrajectory of the dynamical system $f$ if the following inequalities are satisfied
\begin{eqnarray*}
\dist(x_{k+1},f(x_k))& < &d(|k|+1)^{-\gamma},\quad k\in\Zb.
\end{eqnarray*}
\end{deff}

\begin{deff}
We say that the homeomorphism $f$ has Lipschitz two-sided limit shadowing property with exponent $\gamma$ if there exist positive constants $d_0,L$ such that for any $\gamma$-decreasing $d$-pseudotrajectory $\{x_k\}$ with $d\leq d_0$ there exists a point $p\in M$ such that
\begin{eqnarray*}
\dist(x_{k},f^k(p)) &\leq & Ld(|k|+1)^{-\gamma},\quad k\in\Zb.
\end{eqnarray*}
We write $f\in LTSLmSP(\gamma)$ in this case.
\end{deff}

\begin{rem}
In \cite{PILSDS}, where some similar shadowing properties has been studied\hl{, it} is shown that in the neighborhood of a hyperbolic set both $L_p$-shadowing and weighted shadowing (for a special choice of weights) are present.
\end{rem} 
\subsection{Main results}

A diffeomorphism $f$ of a smooth manifold $M$ is said to be structurally stable if there exists a neighborhood $U$ of the diffeomorphism $f$ in the $C^1$-topology such that any diffeomorphism $g\in U$ is topologically conjugate to $f.$

\begin{theorem} \label{thm:shadeqss}
Let $\gamma \geq 0$ and $f$ be a diffeomorphism of a closed Riemannian manifold $M.$ Then $f$ is structurally stable iff $f\in LTSLmSP(\gamma).$
\end{theorem} 
\subsection{Lipschitz two-sided limit shadowing property implies structural stability}

We use one well-known result of R. Mane.
Let $M$ be a closed Riemannian manifold and $f$ be a diffeomorphism of $M.$ We denote the tangent space to the manifold $M$ at a point $p$ by $T_pM.$ Fix a point $p\in M$ and consider two linear subspaces of $T_pM$:
\begin{eqnarray*}
B^+(p)&=& \setdef{v\in T_pM}{\vmod{Df^k(p)v}\to 0,\quad k\to +\infty},\\
B^-(p)&=&\setdef{v\in T_pM}{\vmod{Df^k(p)v}\to 0,\quad k\to -\infty}.
\end{eqnarray*}

\begin{deff}
We say that for a diffeomorphism $f$ the analytical transversality condition is satisfied at a point $p$ if
\begin{equation*}
B^+(p)+B^-(p)=T_pM.
\end{equation*}
\end{deff}

\begin{theoremnon}[Ma\~n\'e, \cite{MANECASD}]
Diffeomorphism $f$ is structurally stable iff the analytical transversality condition is satisfied at every point $p$ of $M.$
\end{theoremnon}

At first we prove one simple lemma
\begin{lemmaa} \label{lemma:finb2infb}
If for a sequence $\seqz{w}{k}$ from $\Nc_\gamma$ there exists a constant $Q$ such that for any integer $N>0$ there exists a sequence $\{v^{N}_k\}_{k\in\Nints}$ of vectors from $\Rb^d$ satisfying equalities
\begin{equation}
v^{N}_{k+1} = A_k v^{N}_k + w_{k+1},\ k\in\Nint, \label{eq:lf2ib}
\end{equation}
and inequalities $\vmod{v^{N}_k}(|k|+1)^\gamma\leq Q,\ k\in\Nints$ then there exists a sequence  $\seqz{v}{k}$ such that it satisfy the same inequalities \eqref{eq:lf2ib} for every integer $k$ and $\normban{\seqz{v}{k}}_\gamma\leq Q.$
\end{lemmaa}

\begin{proof}
To obtain a sequence needed we use a diagonal procedure (we take $v_k^{N}=0,\ k\notin \Nints$) and pass to a limit in inequalities \eqref{eq:lf2ib}. 
\hl{Th}e sequence we get as a result has all the necessary properties.
\end{proof}

\begin{statement}
Let $f$ be a diffeomorphism of a closed Riemannian manifold and let $\gamma > 0.$ If $f\in LTSLmSP(\gamma)$ then $f$  is structurally stable.
\end{statement}

\begin{proof}

Using Theorem \ref{thm:mainpliscor} we show that if we have Lipschitz two-sided limit shadowing property then the analytical transversality condition is satisfied at every point. After that we just apply the Mane theorem.

Fix a point $p\in M,$ denote $p_k=f^k(p)$ and define linear isomorphisms $A_k=Df(p_k)$ for $k\in\Zb.$ We denote a ball in $M$ with a radius $r$ and a center $x$ by $B(r,x)$ and a ball in $T_xM$ with radius $r$ and center $0$ by $B_T(r,x).$

The fact that the norms of all $A_k$ and $\inv{A_k}$ are bounded follows from the compactness of the manifold. We prove that under our assumptions property $B_{\gamma}(\Zb)$ is satisfied for the sequence of matrices $A_k.$ After that we will be able to use Theorem \ref{thm:mainpliscor}.

Let $\exp_x:T_xM\to M$ be a standard exponential mapping. There exists a $r>0$ such that for any point $x\in M$ the mapping $\exp_x$ is a diffeomorphism of a ball $B_T(r,x)$ onto its image and $\exp^{-1}_x$ is a diffeomorphism of a ball $B(r,x)$ onto its image. Moreover, we may assume that the smallness of $r$ allows us to write the following estimates for relations between distances in the manifold and in a tangent space:

\noindent if $v,w\in B_T(r,x)$ then
\begin{equation}
\label{neq:dist2mod}
\mbox{dist}(\exp_x(v),\exp_x(w))\leq 2|v-w|;
\end{equation}
if $y,z\in B(r,x)$ then
\begin{equation}
\label{neq:mod2dist}
|\exp^{-1}_x(y)-\exp^{-1}_x(z)|\leq 2\mbox{dist}(y,z).
\end{equation}

Consider mappings
$$
F_k=\emk\circ f\circ\ek: T_{p_k}M\to T_{p_{k+1}}M.
$$
From the well-known properties of an exponential mapping we deduce that $D\exp_x(0)=\mbox{Id};$ therefore
$$
DF_k(0)=Df(p_k).
$$
Since $M$ is compact for any $\eps>0,$ we can find a $\delta>0$ such that if $|v|\leq\delta,$ then for $g_k(v) = F_k(v)-A_kv$ the following inequality is satisfied
\begin{equation}
\label{neq:compactsmallness}
|g_k(v)|\leq\eps|v|.
\end{equation}

Let $L,d_0$ be the constants from the definition of $LTSLmSP(\gamma).$

We prove that for any sequence of vectors $\seqz{z}{k}\in\Nc_\gamma$ satisfying $\normban{\seqz{z}{k}}_\gamma < 1$ there exists a sequence $\seqz{v}{k}\in\Nc_\gamma$ that is a solution of equations
\begin{equation}
v_{k+1}=A_kv_k + z_{k+1},\quad k\in\Zb. \label{eq:sublinsmethdiffeq}
\end{equation}
After this if we use Theorem \ref{thm:mainpliscor} then we obtain that the analytical transversality condition is satisfied at the point $p.$

We show that the conditions of Lemma \ref{lemma:finb2infb} are satisfied.

We fix natural $N$ and define vectors $a_k:$
\begin{equation}
a_{-N}=0,\quad a_{k+1}=A_ka_k+z_{k+1},\quad k\in\Nint.
\end{equation}

\hl{It is obvious that norms of these vectors are bounded by a constant $C(N),$ that depends only on the norms of the matrices $A_k$ and the number $N.$ Now we fix a positive $d$ satisfying}
\begin{equation} \label{neq:dsmalltaylorwork}
    (8L + 4C(N))d \leq \delta.
\end{equation}

Now we assume that $d$ is small enough so that all the points of $M$ that appear belong to the corresponding balls $B(r,p_k)$ and all tangent vectors from $T_{p_k}M$ that appear belong to the corresponding balls $B_T(r,p_k).$

We define a sequence $\xi_k\in M$ in the following way: let
$\xi_k=\ek(da_k)$ for $|k|\leq N,$ $\xi_{N+k}=f^k(\xi_N)$ for $k>0$
and $\xi_{-N+k}(x)=f^k(\xi_{-N})$ for $k<0.$

We are going to write an estimate for $\dist(f(\xi_k),\xi_{k+1})$ for $ k \in \Nint.$
\hl{Denote} $R=\max_{x\in M} \normban{Df(x)}.$
\hl{Decrease $\eps$ from the inequality }   \eqref{neq:compactsmallness}   \hl{to make it satisfy}
\begin{equation} \label{neq:epsestimate}
    \eps < {(\vmod{k} + 1)^{-\gamma}}  \frac{1}{ (8L + 4C(N)) NR^{2N} }   ,\quad k\in\Nints.
\end{equation}
Decrease also $\delta$ so that inequality \eqref{neq:compactsmallness} still remain true.

Since
$$
\emk(f(\xi_k)) = F_k(da_k)=A_k (da_k) + g_k(da_k)
$$
and
$$
\emk(\xi_{k+1}) = da_{k+1}=d(A_ka_k + z_{k+1}),
$$
after use of estimates \eqref{neq:mod2dist}, \eqref{neq:compactsmallness} and \eqref{neq:epsestimate} we obtain
$$
\dist(f(\xi_k),\xi_{k+1})\leq 2|F_k(da_k)-da_{k+1}|=
$$
$$
= 2|g_k(da_k)-dz_{k+1}|\leq 4d(|k|+1)^{-\gamma}.
$$
For $k \notin \Nint$ the distance $\dist(f(\xi_k),\xi_{k+1})$ equals $0.$ Thus the sequence $\xi_k$ is a $\gamma$-decreasing $4d$-pseudotrajectory. Without loss of generality we can assume that $4d<d_0.$ Since $f\in LTSLmSP(\gamma),$ there exists a 
\hl{trajectory $y_k$ of $f$}
such that
$$
\dist(\xi_k,y_k)\leq 4Ld(|k|+1)^{-\gamma}.
$$
Let $t_k = \eki(y_k).$ Then
$$
t_{k+1} = F_k(t_k) = A_k t_k + g_k(t_k).
$$
%
\hl{Using inequality} \eqref{neq:dsmalltaylorwork} it is easy to see that
$$ \vmod{t_k}\leq 2\dist(y_k,p_k)\leq 2\dist(y_k,\xi_k) + 2\dist(\xi_k,p_k)\leq $$
$$ \leq 8Ld(|k|+1)^{-\gamma} + 4d|a_k| \leq \enbrace{8L + 4C(N)} d,\quad k\in\Nints. $$

 Denote
$$b_k = \Phi_{k,-N}t_{-N},\quad k\in\Nints,$$
$$c_k = t_k - b_k,\quad k\in\Nints.$$
\hl{Then, using inequality}  \eqref{neq:compactsmallness}, estimates of $\vmod{t_k}$ and inequality  \eqref{neq:epsestimate} \hl{we obtain the following:}
$$c_{-N} = 0,\quad c_{k+1} = A_k c_k + g_k(t_k),\quad k\in\Nint,$$
$$\vmod{c_k} = \vmod{ A_{k-1} \enbrace{ A_{k-2} t_{k-2} + g_{k-2}(t_{k-2}) } +g_{k-1}(t_{k-1}) - \Phi_{k,-N}t_{-N}} = $$
%
$$ = \left| \lb \Phi_{k,-N}t_{-N} + g_{k-1}(t_{k-1}) + A_{k-1} g_{k-2}(t_{k-2}) +\ldots + \right. \right. $$
$$ \left. \left. + \Phi_{k-1,-N} g_{-N} (t_{-N}) \rb - \Phi_{k,-N}t_{-N} \right| < N R^{2N} \eps (DL+4C(N)) d <  $$
$$ <d(|k|+1)^{-\gamma},\quad k\in\Nint. $$
Now we show that this is the sequence we have looked \hl{for:}
$$v_k = a_k - \frac{b_k}{d},\quad k\in\Nint,$$
$$v_k = 0,\quad k\notin\Nint.$$
The fact that this sequence is a solution of equations \eqref{eq:sublinsmethdiffeq} is obvious. To estimate its norm in the space $\Nc_\gamma$ by a number independent of $N$ we write \hl{the following:}
$$\vmod{a_k-\frac{b_k}{d}} = \frac{1}{d}\vmod{da_k - b_k} = \frac{1}{d}\vmod{ (\eki(\xi_k) - \eki(y_k) ) + c_k }\leq$$
$$\leq \frac{1}{d}\enbrace{8Ld(|k|+1)^{-\gamma} + d(|k|+1)^{-\gamma} } = (8L+1)(|k|+1)^{-\gamma}.$$

\end{proof}

\subsection{Structural stability implies Lipschitz two-sided limit shadowing property}

We use the method from \cite{PILSDS} to prove that structural stability implies Lipschitz two-sided limit shadowing property. Let $H_k,\ k\in \Zb$ be a sequence of subspaces of $\Rb^d.$ Consider a sequence of linear mappings $\mathcal{A}=\{ A_k:H_k \to H_{k+1} \}.$

\begin{deff} We say that a sequence $\mathcal{A}$ has property (C) with constants $N>1$ and $\lambda\in(0,1)$ if for any integer $k$ there exist projections $P_k, Q_k$ such that if $S_k=P_kH_k$ and $U_k=Q_kH_k$ then the following conditions are satisfied:
\begin{itemize}
\item $P_k + Q_k = Id,\quad \normop{P_k},\normop{Q_k} \leq N;$
\item $A_kS_k\subset S_{k+1}$ and $\normop{A_k\arrowvert_{S_k}}\leq \lambda;$
\item If $U_{k+1}\neq \{ 0 \}$ then there exists a linear mapping $B_k:U_{k+1}\to H_k$ such that
$$ B_kU_{k+1} \subset U_k,\ \normop{B_k}\leq \lambda,\ A_kB_k\arrowvert_{U_{k+1}} = I. $$
\end{itemize}
\end{deff}

The following theorem is a simple modification of Theorem 1.3.1 from \cite{PILSDS}:
\begin{theorem} \label{thm:like131}
Let $\gamma > 0$ and let $\mathcal{A}$ have property (C) with constants $N>1$ and $\lambda\in(0,1).$ Consider a sequence of mappings $f_k:H_k \to H_{k+1}$ of form $f_k(v) = A_k v + w_{k+1}(v).$ Suppose that there exist constants $\kappa,\Delta > 0$ such that the following inequalities are satisfied:
$$
\vmod{\omega_k(v)-\omega_k(v')}\leq \kappa \vmod{v-v'},\quad \vmod{v},\vmod{v'}\leq \Delta,\ k\in \Zb;
$$
$$ \kappa N_1 < 1, $$
with
$$ N_1 = N \frac{1+\lambda}{1-\lambda}. $$
Denote
$$ L = \frac{N_1}{1-\kappa N_1}. $$
Then if
$$ \normban{\{f_k(0)\}}_{\gamma} \leq d \leq d_0, \quad k\in \Zb, $$
with
$$d_0 = \frac{\Delta}{L},$$
then there exists a sequence $v_k\in H_k$ such that $f_k(v_k)=v_{k+1}$ and $\normban{\{v_k\}}_\gamma\leq Ld.$

\end{theorem}

\begin{proof}
Consider an operator $G:\Nc_{\gamma}(\Zb) \to (\Rb^d)^\Zb$ of the form
$$G(z) = \hat{g}_1(z) +  \hat{g}_2(z),$$
\hl{with}
$$ (\hat{g}_1(z))_k = \sum_{u=-\infty}^{k} A_{k-1}\dots A_{u}P_u z_u,$$
$$ (\hat{g}_2(z))_k = -\sum_{u=k+1}^{\infty} B_{k}\dots B_{u-1}Q_u z_u.$$
We prove that the operator $G$ maps $\Nc_{\gamma}(\Zb)$ to $\Nc_{\gamma}(\Zb)$ and is bounded. 
We prove \hl{that} 
$(|k|+1)^{\gamma} \vmod{(G(x))_k} \leq C$
for $k \geq 0$ (for $k\leq 0$ the proof is analogous). We represent $G(z)$ in the following form
$$G(z) = g_1(z) +  g_2(z) + g_3(z)$$
with
$$ (g_1(z))_k = \sum_{u=-\infty}^{0} A_{k-1}\dots A_{u}P_u z_u,$$
$$ (g_2(z))_k = \sum_{u=0}^{k} A_{k-1}\dots A_{u}P_u z_u,$$
$$ (g_3(z))_k = -\sum_{u=k+1}^{\infty} B_{k}\dots B_{u-1}Q_u z_u.$$
Lemma \ref{lemma:seriesbound} allows us to estimate $(k+1)^{\gamma}\vmod{ (g_2(z))_k + (g_3(z))_k }.$ It remains to estimate only $(k+1)^{\gamma} \vmod{(g_1(z))_k}$\hl{:}
$$
(k+1)^{\gamma} \vmod{ (g_1(z))_k } \leq (k+1)^{\gamma} \sum_{u=-\infty}^{0} \lambda^{k-u} (\vmod{u}+1)^{-\gamma} \leq
$$
%
$$
\leq (k+1)^{\gamma} \lambda^k \sum_{u=0}^{\infty} \lambda^{u} \leq C_1.
$$
The rest of the proof is fully analogous to the proof of Theorem 1.3.1 from \cite{PILSDS}.
\end{proof}

Now one can repeat the proof of Theorem 2.2.7 from \cite{PILSDS}, using previous theorem instead of Theorem 1.3.1 from \cite{PILSDS}, where it is necessary. In view of this we have just proved the second implication (the ``only if'' part) of Theorem \ref{thm:shadeqss}.

\begin{theorem}
Let $f$ be a structurally stable diffeomorphism of the closed Riemannian manifold $M.$ Then $f\in LTSLmSP(\gamma)$ for any $\gamma \geq 0.$
\end{theorem}

\section{Acknowledgements}

I am grateful to S. Yu. Pilyugin and S. Tikhomirov for their helpful comments and advices.


\printbibliography

\textit{E-mail address:} todorovdi@gmail.com


\end{document}